\documentclass[11pt,twoside]{amsart}

\usepackage{latexsym}
\usepackage{amsmath}
\usepackage{amssymb}
\usepackage{amsthm}
\usepackage{amscd}
\usepackage[all]{xy}
\usepackage{enumerate} 
\usepackage{amssymb} 
\usepackage{mathrsfs}

\def\N{{\mathscr{N}}}
\def\ZZ{{\mathbb{Z}}}
\def\C{{\mathscr{C}}}

\def\PP{{\mathbb{P}}}
\def\E{{\mathscr{E}}}
\def\Q{{\mathscr{Q}}}

\def\F{{\mathscr{F}}}
\def\GG{{\mathbb{G}}}
\def\S{{\mathscr{S}}}

\usepackage{latexsym}

\newtheorem{them}{Theorem}[section]

\newtheorem{pro}[them]{Proposition}

\newtheorem{exa}[them]{Example}

\newtheorem{lem}[them]{Lemma}

\newtheorem{rem}[them]{Remark}

\newtheorem{cor}[them]{Corollary}

\newtheorem{conj}[them]{Conjecture}

\begin{document}

\title[Fano 5-folds with nef tangent bundles]{Fano 5-folds with nef tangent bundles and Picard numbers greater than one}
\author{Kiwamu Watanabe}
\date{\today}

\address{Course of Mathematics, Programs in Mathematics, Electronics and Informatics, 
Graduate School of Science and Engineering, Saitama University.
Shimo-Okubo 255, Sakura-ku Saitama-shi, 338-8570, Japan.}
\email{kwatanab@rimath.saitama-u.ac.jp}


\maketitle



\begin{abstract}
We prove that smooth Fano 5-folds with nef tangent bundles and Picard numbers greater than one are rational homogeneous manifolds. 

\end{abstract}


\section{Introduction}

Characterization problems of special projective manifolds in terms of positivity properties of the tangent bundle have been considered by several authors. One of the most important results is S. Mori's solution of the Hartshorne-Frankel conjecture \cite{Mori}: a projective manifold with ample tangent bundle is a projective space.

As a generalization of Mori's theorem, F. Campana and T. Peternell  \cite{CP} proposed to study complex projective manifolds with nef tangent bundles and gave the classification in case of dimension $3$. After that, a structure theorem of such manifolds in arbitrary dimension was provided by J. P. Demailly, T. Peternell and M. Schneider \cite{DPS}: a projective (or more generally, compact K${\rm \ddot{a}}$lher) manifold $X$ with nef tangent bundle admits a finite \'etale cover $\tilde{X} \rightarrow X$ such that the Albanese map $\tilde{X} \rightarrow  {\rm Alb}(\tilde{X})$ is a smooth morphism whose fibers are Fano manifolds with nef tangent bundles.

Hence, we obtain the complete picture of projective manifolds with nef tangent bundles if the following conjecture due to Campana and Peternell is solved:  
\begin{conj}[{\cite{CP}}]\label{CPC} A Fano manifold $X$ with nef tangent bundle is rational homogeneous.
\end{conj}

By the classification theory of Fano manifolds, one can check that this conjecture holds when $\dim X \leq 3$. Furthermore, Campana and Peternell \cite{CP2} gave an affirmative answer when $\dim X=4$ and the Picard number $\rho_X>1$. After that, via the works of \cite{CMSB}, \cite{Mi} and \cite{Mok}, the case when $\dim X=4$ was finally completed by J. M. Hwang \cite{Hwang}. However this conjecture remains open in $\dim X \geq 5$. Our main purpose of this article is to treat the case when $\dim X=5$ and $\rho_X>1$.

\begin{them}[=Theorem~\ref{MT2}]\label{MT} 

Let $X$ be a complex Fano manifold of dimension $5$ with nef tangent bundle and Picard number $\rho_X>1$. Then $X$ is a rational homogeneous manifold.
\end{them} 

The proof proceeds as follows. Let $X$ be a Fano $5$-fold with nef tangent bundle of $\rho_X>1$. For any contraction $f: X \rightarrow Y$ of an extremal ray, $f$ is smooth, and $Y$ and the fibers $X_y$ are Fano manifolds with nef tangent bundles (Theorem~\ref{sm}). Furthermore, we see that $\rho_{X_y}=1$. Since Conjecture~\ref{CPC} holds for Fano manifolds of dimension $\leq 4$,  it is easy to see that $X$ is a holomorphic fiber bundle over a rational homogeneous manifold $Y$ whose fibers are projective spaces or quadrics (Lemma~\ref{l}). Since $\rho_X>1$, $X$ admits at least two different fiber bundle structures. Studying these bundle structures, we get the complete classification.

This paper is organized as follows: In Section~$2$, we recall some known results on Fano manifolds. Section~$3$ is dedicated to study properties of Fano manifolds with nef tangent bundles. Furthermore, we shall determine if some concrete examples of Fano manifolds with projective bundle structures have nef tangent bundles. In Section~$4$, we prove our main result Theorem~\ref{MT}. In the final section, we deal with Fano $5$-folds with nef tangent bundles of $\rho=1$.

In this paper, we use notation as in \cite{Ha} and every point on a variety we deal with is a closed point. Denote the $m$ times product of $\PP^n$ by $(\PP^n)^m$. A {\it $\PP^m$-bundle} means the Grothendieck projectivization of a rank $(m+1)$ vector bundle, whereas a smooth morphism whose fibers are isomorphic to $\PP^m$ will be called a {\it smooth $\PP^m$-fibration}. We work over the field of complex numbers. \\

\section{Known results on Fano manifolds}

A {\it Fano manifold} means a projective manifold $X$ with ample anticanonical divisor $-K_X$. For a Fano manifold $X$,  the {\it pseudoindex} is defined as the minimum $i_X$ of the anticanonical degrees of rational curves on $X$.

 Given a projective manifold $X$, we denote by $N_1(X)$ the space of $1$-cycles with real coefficients modulo numerical equivalence. The dimension of $N_1(X)$ is the Picard number $\rho_X$ of $X$.  The convex cone of effective $1$-cycles in $N_1(X)$ is denoted by $NE(X)$. 
By the Contraction Theorem, given a $K_X$-negative extremal ray $R$ of the Kleiman-Mori cone $\overline{NE}(X)$, we obtain the contraction of the extremal ray $\varphi_R :X \rightarrow Y$. 
We say that $\varphi_R$ is {\it of fiber type} if $\dim X > \dim Y$, otherwise it is {\it of birational type}.

\begin{pro}[{\cite[Lemma~3.3, Remark~3.7]{Casa}}]\label{Casa} Let $X$ be a Fano manifold, $f: X \rightarrow Y$ a contraction of an extremal ray of fiber type, and $X_y$ a fiber of $f$. Suppose that $f$ is smooth. Then $X_y$ is a Fano manifold of $\rho_{X_y}=1$.
\end{pro}

\begin{pro}[{\cite[Lemma~4.1]{NO}}]\label{NO2} Let $X$ be a Fano manifold admitting a $\PP^r$-bundle structure $f: X \rightarrow Y$ and $R$ the extremal ray corresponding to $f$. If there exists a proper morphism $g: X \rightarrow Z$ onto a variety $Z$ of dimension $r$ which does not contract curves of $R$. Then $X \cong \PP^r \times Y$  
\end{pro}

\begin{pro}[See {\cite[Proposition~5.1]{NO}} and {\cite[Proposition~2.4]{BCDD}}]\label{NO} Let $X$ be a Fano manifold of dimension $n$ and pseudoindex $\geq 2$ which has only contractions of fiber type. 
Then $\rho_X \leq n$. Moreover, 
\begin{enumerate}
\item if $\rho_X=n$, then $X=(\PP^1)^n$;
\item if $\rho_X=n-1,$ then $X$ is either $(\PP^1)^{n-2} \times \PP^2$ or $X=(\PP^1)^{n-3} \times \PP(T_{\PP^2})$.
\end{enumerate}
\end{pro}

\begin{rem} \rm The above result {\cite[Proposition~5.1]{NO}}  was obtained by applying {\cite[Lemma~2.13]{NO}}. As one of referees pointed out to the author, the proof of {\cite[Lemma~2.13]{NO}} contains a gap. To be more precise, it is based on a result in algebraic topology due to A. Borel \cite[Expos$\rm \acute{e}$ IX, Remark 2 after Theorem 6]{Borel}, and it seems that the Borel's proof works when both $H^{\ast}(F, \ZZ)$ and $H^{\ast}(B, \ZZ)$ are torsion-free. 

However it does not affect {\cite[Proposition~5.1]{NO}}. Under the notation as in {\cite[Proposition~5.1]{NO}}, the Borel's result was applied to prove that every elementary contraction $\varphi_j: X \rightarrow Y_j$ with one-dimensional fibers is given by the projectivization of a rank $2$ vector bundle. Without using \cite[Expos$\rm \acute{e}$ IX, Remark 2 after Theorem 6]{Borel}, by the same way as in the proof of {\cite[Proposition~5.1]{NO}}, we see that $Y_j \cong (\PP^1)^{n-3} \times \PP^2$ or $(\PP^1)^{n-4} \times \PP(T_\PP^2)$. Then it follows from Proposition~\ref{Br} below that $\varphi_j$ is a $\PP^1$-bundle. As a consequence, we obtain {\cite[Proposition~5.1]{NO}} by the same argument.     
\end{rem}

\begin{pro}\label{Br} Let $f: X \rightarrow Y$ be a smooth $\PP$-fibration over a projective manifold $Y$. If $Y$ is rational or a curve, then there exists a rank $(d+1)$ vector bundle $\E$ on $Y$ such that $X=\PP_Y(\E)$. 
\end{pro}

\begin{proof} Consider an exact sequence of algebraic groups over $Y$:
\begin{eqnarray}
1 \rightarrow \GG_m \rightarrow GL(d+1) \rightarrow PGL(d) \rightarrow 1. \nonumber
\end{eqnarray}
Then we have an exact sequence of $\rm \acute{e}$tale cohomologies:
\begin{eqnarray}\label{exact}
 H^1_{{\it {\acute{e}}t}}(Y, GL(d+1)) \rightarrow H^1_{{\it {\acute{e}}t}}(Y, PGL(d)) \rightarrow H^2_{{\it {\acute{e}}t}}(Y, \GG_m).
\end{eqnarray}

Here ${\rm Br}'(Y):=H^2_{\it {\acute{e}}t}(Y, \GG_m)$ is called the {\it Brauer-Grothendieck group} of $Y$. The Brauer-Grothendieck group is birational invariant of complex projective manifolds \cite[III, Corollary~7.3]{Gr}. Furthermore, it is well-known that ${\rm Br}'(Y)$ is trivial when $Y$ is a complex projective space or a curve. Hence, in the above sequence (\ref{exact}), the first arrow is surjective. 
 
On the other hand, a smooth $\PP$-fibration $f$ defines a cocycle $[f] \in H^1_{{\it {\acute{e}}t}}(Y, PGL(d))$. Then $f$ is given by the projectivization of a vector bundle if and only if there exists a preimage of $[f]$ in $H^1_{{\it {\acute{e}}t}}(Y, GL(d+1))$. Since the first arrow of the above sequence (\ref{exact}) is surjective, we obtain our assertion.    
\end{proof}

\begin{pro}[{\cite[Theorem~2]{OW}}]\label{WO} Let $X$ be a projective manifold of dimension $n$, endowed with two different smooth $\PP$-fibration structures $f: X \rightarrow Y$ and $g: X \rightarrow Z$ such that $\dim Y + \dim Z = n+1$. Then either $n=2m-1$, $Y =Z=\PP^m$ and $X=\PP(T_{\PP^m})$ or $Y$ and $Z$ have a $\PP$-bundle structure over a smooth curve $C$ and $X = Y \times_C Z$.
\end{pro}

\begin{proof} See {\cite[Theorem~2]{OW}}. According to Proposition~\ref{Br}, a smooth $\PP$-fibration over a curve is a $\PP$-bundle.
\end{proof}

\section{Fano manifolds with nef tangent bundles}

\begin{them}[{See \cite[Theorem~4.2]{Hwang}}]\label{4} Let $X$ be a Fano manifold with nef tangent bundle of dimension $n \leq 4$. Then the following holds.
\begin{enumerate}
\item If $n=1$, then $X$ is $\PP^1$.
\item If $n=2$, then $X$ is $\PP^2$ or $(\PP^1)^2$.
\item If $n=3$, then $X$ is  one of the following: \\
$\PP^3$, $Q^3$, $\PP^1 \times \PP^2$, $\PP(T_{\PP^2})$, $(\PP^1)^3$.
\item If $n=4$, then $X$ is  one of the following: \\
$\PP^4$, $Q^4$, $\PP^1 \times \PP^3$, $\PP^1 \times Q^3$, $(\PP^2)^2$, $\PP(\N)$, where $\N$ is the null-correlation bundle over $\PP^3$ (see Example~\ref{spe} below), $(\PP^1)^2 \times \PP^2$, $\PP^1 \times \PP(T_{\PP^2})$, $(\PP^1)^4$.
\end{enumerate}
\end{them}

\begin{proof} When $n \leq 2$, it is easy to prove our assertion. When $n=3$, this is in \cite[Theorem~5.1, Theorem~6.1]{CP}. Of course, this also follows from the classification theory of Fano manifolds of $n \leq 3$. If $n=4$ and $\rho_X > 1$, then our assertion is dealt in \cite[Theorem~3.1]{CP2}. However we should remark that the tangent bundle of  $\PP(T_{\PP^2})\times_{\PP^2} \PP(T_{\PP^2})$, which is listed in \cite[Theorem~3.1~(4)-(d)]{CP2}, is not nef, see Lemma~\ref{non} below. If $n=4$ and $\rho_X =1$, we see that $X$ is isomorphic to $\PP^4$ or $Q^4$. This follows from \cite{CMSB}, \cite{Mi} and \cite[Theorem~4.3]{Hwang} (see also Section~\ref{1}).  

\end{proof}

\begin{lem}\label{pi} Let $X$ be a Fano manifold with nef tangent bundle. Then the pseudoindex of $X$ is at least $2$.
\end{lem}

\begin{proof} Let $C$ be a rational curve on $X$ and $f: \PP^1 \rightarrow C \subset X$ its normalization. Since $T_X$ is nef, so is $f^{\ast}T_X$. This implies that $f^{\ast}T_X \cong \bigoplus^n_{i=1} {\mathscr{O}}_{\PP^1}(a_i)$, where $a_i \geq 0$. Furthermore we have an injection ${\mathscr{O}}_{\PP^1}(2) \rightarrow f^{\ast}T_X$. This implies that $a_i \geq 2$ for some $i$. Consequently, $-K_X.C= \sum a_i \geq 2$. This means the pseudoindex of $X$ is at least $2$.

\end{proof}

\begin{lem}\label{non} The tangent bundle of $\PP(T_{\PP^2})\times_{\PP^2} \PP(T_{\PP^2})$ is not nef.
\end{lem}

\begin{proof} For $X:=\PP(T_{\PP^2})\times_{\PP^2} \PP(T_{\PP^2})$, consider the commutative diagram: \[\xymatrix{
&& X \ar[dd]_{\pi} \ar[dl]_{\pi_1} \ar[dr]^{\pi_2} && \\
&\PP(T_{\PP^2}) \ar[dl]_{p_1} \ar[dr]^{p_2}  & & \PP(T_{\PP^2}) \ar[dl]_{p_2} \ar[dr]^{p_1} & \\
\PP^2 & &\PP^2&  & \PP^2 \\
}\]

Then we have 
\begin{eqnarray} \nonumber
-K_X&=& \pi^{\ast}(-K_{\PP^2})+(-K_{X/\PP^2})\\
\nonumber &=& \pi^{\ast}(-K_{\PP^2})+\pi_1^{\ast}(-K_{\PP(T_{\PP^2})/\PP^2})+ \pi_2^{\ast}(-K_{\PP(T_{\PP^2})/\PP^2})\\ 
\nonumber &=& \pi^{\ast}(-K_{\PP^2})+\pi_1^{\ast}(-K_{\PP(T_{\PP^2})}+p_2^{\ast}(K_{\PP^2}))+ \pi_2^{\ast}(-K_{\PP(T_{\PP^2})}+p_2^{\ast}(K_{\PP^2}))\\ 
\nonumber &=& \pi^{\ast}(K_{\PP^2})+\pi_1^{\ast}(-K_{\PP(T_{\PP^2})})+ \pi_2^{\ast}(-K_{\PP(T_{\PP^2})}).
\end{eqnarray}

Let $l \subset \PP(T_{\PP^2})$ be a fiber of $p_1$. Then ${p_2}_{\ast}(l)$ is a line in $\PP^2$. Furthermore, $l$ can be regarded as a curve in $X$ via the diagonal embedding $\PP(T_{\PP^2}) \subset X$. Then 
\begin{eqnarray} \nonumber
-K_X.l= (\pi^{\ast}(K_{\PP^2})+\pi_1^{\ast}(-K_{\PP(T_{\PP^2})})+ \pi_2^{\ast}(-K_{\PP(T_{\PP^2})})).l=1.
\end{eqnarray}    
Thus, Lemma~\ref{pi} concludes that the tangent bundle of $X$ is not nef.   
\end{proof}

\begin{rem}\rm  We see that $\PP(T_{\PP^2})\times_{\PP^2} \PP(T_{\PP^2})$ is the blow-up of $\PP^2 \times \PP^2$ along the diagonal. Lemma~\ref{non} also follows from this fact (see Theorem~\ref{sm}~$\rm (i)$ below). 

\end{rem}

\begin{them}\label{sm} Let $X$ be a Fano manifold with nef tangent bundle, $f: X \rightarrow Y$ a contraction of an extremal ray and $X_y$ a fiber of $f$. Then the following holds.
\begin{enumerate}
\item $f$ is smooth, in particular, of fiber type. 
\item $Y$ is a Fano manifold with nef tangent bundle of $\rho_Y=\rho_X-1$.
\item $X_y$ is a Fano manifold with nef tangent bundle of $\rho_{X_y}=1$.
\end{enumerate}
\end{them}

\begin{proof} $\rm (i)$ This is in \cite[Theorem~5.2]{DPS} (see also \cite[Theorem~4.4]{SolW}). 

$\rm (ii)$ An image of a Fano manifold by a smooth morphism is again Fano (see \cite[Corollary~2.9]{KMM}). Furthermore, it follows from \cite[Proposition~2.11 (2)]{CP} that $T_Y$ is nef. 

$\rm (iii)$ From Proposition~\ref{Casa}, it follows that $X_y$ is a Fano manifold of $\rho_{X_y}=1$. Moreover \cite[Proposition~2.11 (1)]{CP} implies that $T_{X_y}$ is nef. 
\end{proof}

\begin{pro}\label{normal} Let $X$ be a Fano manifold with nef tangent bundle, $f: X \rightarrow Y$ a contraction of an extremal ray and $F$ a projective submanifold of $Y$ whose normal bundle is trivial, i.e., $N_{F/Y} \cong {\mathscr{O}}_F^{\oplus l}$. Then the preimage $W:=f^{-1}(F)$ is a Fano manifold with nef tangent bundle. 
\end{pro}

\begin{proof} By \cite[II.~Proposition~8.10]{Ha}, we see that $T_{W/F} \cong T_{X/Y}|_W$. So we have the following exact commutative diagram:
\[\xymatrix{ &&0&0& \\
& 0 & N_{W/X} \ar[u] & f_W^{\ast}(N_{F/Y}) \ar[u] & \\
0 \ar[r] & T_{X/Y}|_W \ar[r] \ar[u] & T_X|_W \ar[r] \ar[u] & f^{\ast}(T_Y)|_W \cong f_W^{\ast}(T_Y|_F) \ar[r] \ar[u] & 0 \\
0 \ar[r] & T_{W/F} \ar[r] \ar[u] & T_W \ar[r] \ar[u] & f_W^{\ast}(T_F) \ar[r] \ar[u] & 0 \\ 
&0 \ar[u]&0\ar[u]&0\ar[u]& \\
} \]
Thus the snake lemma implies that $N_{W/X} \cong f_W^{\ast}(N_{F/Y})$. By our assumption, we obtain $N_{W/X} \cong {\mathscr{O}}_W^{\oplus l}$. Then it follows in a similar way to Theorem~\ref{sm}~$\rm (iii)$ that $T_W$ is nef. Furthermore, the adjunction formula tells us that $-K_W=(-K_X)|_W$. This means that $W$ is also a Fano manifold.
\end{proof}

\begin{exa}[Spinor bundle and Null-correlation bundle]\label{spe} \rm Let denote the null-correlation bundle on $\PP^3$ by $\N$ (see \cite[Chapter~1, Section~4.2]{OSS} for the definition). Denote by $\S$ the spinor bundle on $Q^3$, by $\S_1$ and $\S_2$ the two spinor bundles on $Q^4$ (see  \cite[Definition~1.3]{Ot}). 

Then it is known that $\PP(\N)$ and $\PP(\S)$ coincides with the full-flag manifold of type $B_2$. In particular, $\PP(\N)=\PP(\S)$ is a homogeneous manifold. 

On the other hand, the two spinor bundles $\S_1$ and $\S_2$ on $Q^4$ are the universal bundle and the dual of the quotient bundle (see \cite[Example~1.5]{Ot}). Thus, $\PP(\S_1)$ and $\PP(\S_2)$ are isomorphic to the flag manifold $F(1, 2, \PP^3)$ parametrizing pairs $(l, P)$, where $l$ is a line in a plane $P \subset \PP^3$. In particular, $\PP(\S_1) \cong \PP(\S_2)$ is a homogeneous manifold.   

\end{exa}

For a smooth quadric $Q^4$ of dimension $4$, let $H$ be a hyperplane section, and let $P_1$ and $P_2$ be planes in $Q^4$ whose numerical classes are different. Then we have $H^2(Q^4, \ZZ)=\ZZ[H]$ and $H^4(Q^4, \ZZ)= \ZZ[P_1] \oplus \ZZ[P_2]$. By these descriptions, we regard an element of $H^2(Q^4, \ZZ)$ (reap. $H^4(Q^4, \ZZ)$) as one of $\ZZ$ (reap. $\ZZ \oplus \ZZ$).

\begin{lem}\label{f} 
Let $\F$ be a rank $2$ stable vector bundle on $Q^4$ with Chern classes $c_1=-1$ and $c_2=(1, 1)$.  
Then the tangent bundle of $\PP(\F)$ is not nef.
\end{lem}

\begin{proof} According to \cite[Remark~3.4]{Ota}, $\F$ extends to $Q^5$ to a Cayley bundle $\C$. Cayley bundles are characterized by their Chern classes among rank $2$ stable bundles on $Q^5$ (see \cite[Main Theorem]{Ota}). Let $K(G_2)$ be the $5$-dimensional contact homogeneous manifold of type $G_2$. It is known that $K(G_2)$ is a linear section of the Grassmannian $G(1, \PP^6)$ with a $\PP^{13}$. For the restriction of the universal quotient bundle $\Q$ on $G(1, \PP^6)$, we see that $\PP(\Q|_{K(G_2)})$ coincides with $\PP(\C)$. Then it follows from \cite[1.3]{Ota} that $K(G_2)$ is the variety of special lines in $Q^5$ and $\PP(\C)=\PP(\Q|_{K(G_2)})$ is its flag variety $\{(p, l)| p \in l, l~{\rm special~line~in~}Q^5\}$: 
 \[\xymatrix{
& \PP(\C)=\PP(\Q|_{K(G_2)})  \ar[dl]_{p_1} \ar[dr]^{p_2} & \\
Q^5  & & K(G_2)  \\
}\] 
Since $Q^4$ is a hyperplane section of $Q^5$, the restriction map $p_2|_{\PP(\F)}: \PP(\F) \rightarrow K(G_2)$ is surjective. 
Furthermore, by \cite[Theorem~3.5]{Ota} and its proof, it turns out that $Q^4 \subset Q^5$ contains a special line $l_0$ in $Q^5$. It implies that $p_2|_{\PP(\F)}$ has a positive-dimensional fiber. By taking the Stein factorization, one can factor $p_2|_{\PP(\F)}$ into $g \circ f$, where $f$ is a projective morphism with connected fibers, and $g$ is a finite morphism.
Since $p_2|_{\PP(\F)}$ has a positive-dimensional fiber and $\PP(\F)$ is a Fano manifold (see \cite[Example~2.2]{APW}), $f$ is a contraction of an extremal face. 

If the tangent bundle of $\PP(\F)$ would be nef, then it follows from Theorem~\ref{sm} that $f$ is of fiber type. However it contradicts to $\dim \PP(\F)=\dim K(G_2)$.

\end{proof}

\begin{pro}\label{fb} Let $X$ be a Fano $5$-fold with nef tangent bundle which admits a $\PP^1$-bundle structure $f: X \rightarrow Y$. Let $\N$ be the null-correlation bundle on $\PP^3$, $\S$ the spinor bundle on $Q^3$  and $\S_i$ $(i=1, 2)$ the spinor bundles on $Q^4$ as in Example~\ref{spe}. Then the following holds.
\begin{enumerate}
\item If $Y$ is $\PP^4$, then  $X$ is $\PP^1 \times \PP^4$.
\item If $Y$ is $Q^4$, then $X$ is $\PP^1 \times Q^4$ or $\PP(\S_i)$.
\item If $Y$ is $\PP^1 \times \PP^3$ (resp. $\PP^1 \times Q^3$), then $X$ is $(\PP^1)^2 \times \PP^3$ or $\PP^1 \times \PP(\N)$ (resp. $(\PP^1)^2 \times Q^3$ or $\PP^1 \times \PP(\S)$).
\item If $Y$ is $\PP(\N)$, then $X$ is $\PP^1 \times\PP(\N)$. 
\item If $Y$ is $(\PP^2)^2$, then $X$ is $\PP^1 \times (\PP^2)^2 $ or $\PP^2 \times \PP(T_{\PP^2})$.
\end{enumerate}
In particular, every manifold appeared in the above list is rational homogeneous.
\end{pro}

\begin{proof} Let $\E$ be a rank $2$ vector bundle on $Y$ such that $X=\PP(\E)$.

$\rm (i)$ If $Y$ is $\PP^4$, then it follows from \cite[Main Theorem~2.4]{APW} that $\E$ splits into a direct sum of line bundles as ${\mathscr{O}}_Y(a) \oplus {\mathscr{O}}_Y(b)$. If $a$ is not equal to $b$, then $Y$ has a contraction of birational type. However this contradicts to Theorem~\ref{sm}~{\rm (i)}. Hence $X$ is $\PP^1 \times \PP^4$. 

$\rm (ii)$ If $Y$ is $Q^4$, then \cite[Main Theorem~2.4]{APW} and Lemma~\ref{f} imply that $X$ is $\PP^1 \times Q^4$ or $\PP(\S_i)$, via  the same argument as in $\rm (i)$.

$\rm (iii)$ Let $Y$ be $\PP^1 \times V$, where $V$ is $\PP^3$ or $Q^3$. Let $p_1$ be the first projection $Y \rightarrow \PP^1$ and $p_2$ the second projection $Y \rightarrow V$: 
 \[\xymatrix{
& Y  \ar[dl]_{p_1} \ar[dr]^{p_2} & \\
\PP^1  & & V  \\
}\] 
Let $l$ be a fiber of $p_2$.  
According to Proposition~\ref{normal}, $\PP(\E|_{l})$ is a Fano surface with nef tangent bundle. Thus, by Theorem~\ref{4}, we see that $\E|_{l} \cong {\mathscr{O}}_{\PP^2} \oplus {\mathscr{O}}_{\PP^2}$ up to a twist by a line bundle. Thus, by tensoring a line bundle, we may assume that $\E|_{l} \cong {\mathscr{O}}_{\PP^2} \oplus {\mathscr{O}}_{\PP^2}$ for every fiber $l$ of $p_2$. By applying Grauert's theorem \cite[III. Corollary~12.9]{Ha}, we see that ${p_2}_\ast(\E)$ is a rank $2$ vector bundle on $V$. Furthermore, there is a natural map ${p_2}^{\ast}({p_2}_\ast(\E)) \rightarrow \E$.  For $y \in l$, we have ${p_2}^{\ast}({p_2}_\ast(\E)) \otimes k(y)  \cong H^0(l, \E|_l)$. Again, this follows from Grauert's theorem \cite[III. Corollary~12.9]{Ha}. Hence ${p_2}^{\ast}({p_2}_\ast(\E)) \otimes k(y) \rightarrow \E \otimes k(y)$ is surjective. By Nakayama's lemma, ${p_2}^{\ast}({p_2}_\ast(\E))_y \rightarrow \E_y$ is also surjective, hence, so is ${p_2}^{\ast}({p_2}_\ast(\E)) \rightarrow \E$.
As a consequence, it turns out that 
\begin{eqnarray} 
{p_2}^{\ast}({p_2}_\ast(\E)) \cong \E.  \nonumber
\end{eqnarray}  
For a fiber $F$ of $p_1$, ${p_2}_{\ast}(\E)\cong {p_2}^{\ast}({p_2}_\ast(\E))|_F \cong \E|_F$. This implies that $\E \cong {p_2}^{\ast}({p_2}_\ast(\E)) \cong  {p_2}^{\ast}(\E|_F)$. Thus, we see that $X \cong \PP^1 \times \PP(\E|_F)$. By Proposition~\ref{normal}, $\PP(\E|_F)$ is a Fano $4$-fold with nef tangent bundle. According to Theorem~\ref{4}, if $F \cong \PP^3$ (resp. $F \cong Q^3$), then $\PP(\E|_F)$ is $\PP^1 \times \PP^3$ or $\PP(\N)$ (resp. $\PP^1 \times Q^3$ or $\PP(\S))$. Hence our assertion holds.

 $\rm (iv)$ Let $Y$ be $\PP(\N)$ and $p: \PP(\N) \rightarrow \PP^3$ the bundle projection. By a similar argument to $\rm (iii)$, one can show that $\E|_l={\mathscr{O}}_{\PP^1} \oplus {\mathscr{O}}_{\PP^1}$ for a fiber $l$ of $p$, and $\E=p^{\ast}(\E_0)$ for $\E_0:=p_{\ast}(\E)$. Now we have a base change diagram
\[\xymatrix{
\PP(\E) \ar[r]^{} \ar[d] &  \PP(\E_0) \ar[d]  \\
\PP(\N) \ar[r]^{p} & \PP^3  \\
} \]
Since $X=\PP(\E)$ is a $\PP^1$-bundle over $\PP(\E_0)$, $\PP(\E_0)$ is a Fano $4$-fold with nef tangent bundle. Moreover $\PP(\E_0)$ is a $\PP^1$-bundle over $\PP^3$. Thus, by Theorem~\ref{4}, $\PP(\E_0)$ is $\PP^1 \times \PP^3$ or $\PP(\N)$. This implies that $X$ is $\PP^1 \times \PP(\N)$ or $\PP(\N)\times_{\PP^3} \PP(\N)$. In the later case, we can show that the tangent bundle of $X$ is not nef in a similar way to Lemma~\ref{non}. Indeed, $\PP(\N)$ admits a $\PP^1$-bundle structure over $Q^3$ and denote its fiber by $l$. Remark that $l$ can be regarded as a curve in $X:=\PP(\N)\times_{\PP^3} \PP(\N)$ via the diagonal embedding $\PP(\N) \subset X$. Then we see that $-K_X.l=0$. This implies that $X$ is not Fano. Hence our assertion holds.

$\rm (v)$ Let $Y$ be $(\PP^2)^2$ and $p_i$ the $i$-th projection $Y \rightarrow \PP^2$ ($i=1$, $2$). Let $F_i$ be a fiber of $p_i$. According to Proposition~\ref{normal}, $\PP(\E|_{F_i})$ is a Fano manifold with nef tangent bundle. Thus, by Theorem~\ref{4}, we see that $\E|_{F_i} \cong {\mathscr{O}}_{\PP^2} \oplus {\mathscr{O}}_{\PP^2}$ or $T_{\PP^2}(-1)$, up to a twist by a line bundle. If $\E|_{F_i} \cong {\mathscr{O}}_{\PP^2} \oplus {\mathscr{O}}_{\PP^2}$ for some $i$, then we see that ${p_i}^{\ast}({p_i}_\ast(\E)) \cong \E$ in a similar way to $\rm (iii)$. Furthermore, $\E \cong {p_i}^{\ast}(\E|_{F_j})$ and $\E|_{F_j} \cong {\mathscr{O}}_{\PP^2} \oplus {\mathscr{O}}_{\PP^2}$ or $T_{\PP^2}(-1)$ for $j \neq i$. As a consequence, $X$ is $\PP^1 \times (\PP^2)^2 $ or $\PP^2 \times \PP(T_{\PP^2})$. On the other hand, assume that $\E|_{F_i} \cong T_{\PP^2}(-1)$ for $i=1, 2$. Then $c_1(\E)=(1, 1)$. This implies that ${\mathscr{O}}_X(-K_X) \cong {\mathscr{O}}_{\PP(\E)}(2) \otimes f^{\ast}{\mathscr{O}}_{\PP^2 \times \PP^2}(2,2)$, where ${\mathscr{O}}_{\PP(\E)}(1)$ is the tautological invertible sheaf of $X=\PP(\E)$. This implies that the Fano index of $X$ is $2$. According to Theorem~\ref{sm}, $X$ has only contractions of fiber type. Thus,
it follows from \cite[Proposition~7.1]{NO} that $X$ is a product with $\PP^1$ as a factor. However, this contradicts to $\E|_{F_i} \cong T_{\PP^2}(-1)$.

\end{proof}

\begin{pro}\label{P} Let $X$ be a Fano $5$-fold with nef tangent bundle. Then $\rho_X \leq 3$ or $X$ is one of the following:\\
$(\PP^1)^5$, $(\PP^1)^3 \times \PP^2$, $(\PP^1)^2 \times \PP(T_{\PP^2})$. 
\end{pro}

\begin{proof} By Lemma~\ref{pi}, the pseudoindex of $X$ is at least $2$. Moreover, $X$ has only contractions of fiber type because of Theorem~\ref{sm}. Thus, by applying Proposition~\ref{NO}, we get our assertion.
\end{proof}

\section{Proof of Theorem~\ref{MT}}

Let $\N$ be the null-correlation bundle on $\PP^3$, $\S$ the spinor bundle on $Q^3$  and $\S_i$ $(i=1, 2)$ the spinor bundles on $Q^4$ as in Example~\ref{spe}. In this section, we prove Theorem~\ref{MT}:

\begin{them}[=Theorem~\ref{MT}]\label{MT2} 

Let $X$ be a Fano manifold of dimension $5$ with nef tangent bundle and Picard number $\rho_X>1$. Then $X$ is one of the following:\\
$\PP^1 \times \PP^4$, $\PP^1 \times Q^4$, $\PP^2 \times \PP^3$, $\PP^2 \times Q^3$, $\PP(T_{\PP^3})$, $\PP(\S_i)$, $\PP^1 \times (\PP^2)^2$, $(\PP^1)^2 \times \PP^3$, $(\PP^1)^2  \times Q^3$, $\PP^2 \times \PP(T_{\PP^2})$, $\PP^1 \times \PP(\N)=\PP^1 \times \PP(\S)$,
$(\PP^1)^3 \times \PP^2$, $(\PP^1)^2 \times \PP(T_{\PP^2})$, $(\PP^1)^5$.  

In particular, $X$ is a rational homogeneous manifold.
\end{them} 

Let $X$ be a Fano $5$-fold with nef tangent bundle of $\rho_X \geq 2$. Then there exist two different contractions $f: X \rightarrow Y$ and $g: X \rightarrow 
Z$ of extremal rays:
\[\xymatrix{
X \ar[r]^{f} \ar[d]_{g} & Y \\
Z &  \\
} \]
Denote by $X_y$ (resp. $X_z$) a fiber of $f$ (resp. one of $g$). We may assume that $\dim Z \geq \dim Y (\geq 1)$.  

\begin{lem}\label{l} Under the above setting, the following holds.
\begin{enumerate}
\item $\rho_Y=\rho_Z$.
\item $Y$ and $Z$ are rational homogeneous manifolds listed in Theorem~\ref{4}. Furthermore, $X_y$ and $X_z$ are either $\PP^d$ {\rm (}$1 \leq d \leq 4${\rm )} or $Q^d$ {\rm (}$d=3$ or $4${\rm )}.
\item $5>\dim Y \geq \dim X_z$ and $5>\dim Z \geq \dim X_y$.
\item If $\dim Z = \dim X_y$ and $X_y \cong \PP^d$ (resp. $\dim Y = \dim X_z$ and $X_z \cong \PP^d$), then we have $X \cong \PP^d \times Y$ (resp. $\PP^d \times Z$).  
\item If $\dim Z = \dim X_y$ and $X_y \cong Q^d${\rm (}$d=3$ or $4${\rm )} (resp. $\dim Y = \dim X_z$ and $X_z \cong Q^d$), then $Z$ (resp. $Y$) is either $\PP^d$ or $Q^d$ and $X$ is a $\PP^{5-d}$-bundle over $Z$ (resp. $Y$).  
\end{enumerate}
\end{lem}

\begin{proof} $\rm (i)$ Since $f$ and $g$ are contractions of extremal rays, $\rho_Y=\rho_X-1=\rho_Z$.

$\rm (ii)$ From Theorem~\ref{sm}, $Y$, $Z$, $X_y$ and $X_z$ are Fano manifolds with nef tangent bundles, and $\rho_{X_y}=\rho_{X_z}=1$. Hence Theorem~\ref{4} implies our assertion.

$\rm (iii)$ Since $f$ and $g$ are different contractions, $X_y$ and $X_z$ are not contracted by $g$ and $f$, respectively. Furthermore, we have $\rho_{X_y}=\rho_{X_z}=1$. This implies that $\dim Y \geq \dim X_z$ and $\dim Z \geq \dim X_y$. 

$\rm (iv)$ If $\dim Z = \dim X_y$ and $X_y \cong \PP^d$, then our claim follows from Proposition~\ref{Br} and Proposition~\ref{NO2}.

$\rm (v)$  We see that $Z \cong \PP^d$ or $Q^d$ by \cite[Proposition~8]{PS},  and it follows from $\rm (ii)$ and Proposition~\ref{Br} that $X$ is a $\PP^{5-d}$-bundle over $Z$.

\end{proof}

\subsection{Case where $\dim Y=1$ }

\begin{pro}\label{l1} If $\dim Y=1$, then $X$ is $ \PP^1 \times \PP^4$ or $\PP^1 \times Q^4$. 
\end{pro}

\begin{proof} By Lemma~\ref{l}~$\rm (ii)$, $Y \cong \PP^1$ and $X_y \cong \PP^4$ or $Q^4$. Furthermore, it follows from Lemma~\ref{l}~$\rm (iii)$ that $\dim Z=\dim X_y =4$. If $X_y \cong \PP^4$, then Lemma~\ref{l}~$\rm (iv)$ concludes that $X \cong \PP^1 \times \PP^4$. On the other hand, if $X_y \cong Q^4$, then Lemma~\ref{l}~$\rm (v)$ tells us that $X$ is a $\PP^{1}$-bundle over $Z$. Then, using Proposition~\ref{NO2}, we see that $X \cong \PP^1 \times Q^4$.  
\end{proof} 

\subsection{Case where $\dim Y=2$ }

\begin{pro} If $\dim Y=2$, then $X \cong \PP^2 \times \PP^3$, $\PP^2 \times Q^3$, $(\PP^1)^2 \times \PP^3$ or $(\PP^1)^2  \times Q^3$.  
\end{pro}

\begin{proof}
By Lemma~\ref{l}, we see that $Y \cong \PP^2$ or $(\PP^1)^2$, $X_y \cong \PP^3$ or $Q^3$ and $\dim Z=3$ or $4$, in a similar way to Proposition~\ref{l1}.

If $Y \cong \PP^2$ and $\dim Z=3$, then we have $\dim Y= \dim X_z$ and it follows from Lemma~\ref{l}~$\rm (ii)$ that $X_z \cong \PP^2$. Therefore Lemma~\ref{l}~$\rm (iv)$ implies that $X \cong \PP^2 \times \PP^3$ or $\PP^2 \times Q^3$.

If $Y \cong \PP^2$ and $\dim Z=4$, then $X$ is a $\PP^1$-bundle over $\PP^4$ or $Q^4$ by Lemma~\ref{l}~$\rm (ii)$ and Proposition~\ref{Br}. Therefore we are in the situation of Proposition~\ref{fb}~$\rm (i)$ and $\rm (ii)$. However every manifold appeared there has no contractions to $\PP^2$. Hence we get a contradiction.

If $Y \cong (\PP^1)^2$, then it follows from Lemma~\ref{l}~$\rm (i)$ that $\rho_Z=\rho_Y=2$. By virtue of Lemma~\ref{l}~${\rm (ii)}$, $X_y \cong \PP^3$ or $Q^3$. If $\dim Z=3$, then $Z$ would be isomorphic to $\PP^3$ or $Q^3$ by Lemma~\ref{l}~$\rm (iv)$ and $\rm (v)$. This contradicts to $\rho_Z=2$. Hence $\dim Z=4$. Then, it follows from Lemma~\ref{l}~$\rm (ii)$ and Proposition~\ref{Br} that $X$ is a $\PP^1$-bundle over $Z$, where $Z \cong \PP^1 \times \PP^3,  \PP^1 \times Q^3$, $(\PP^2)^2$ or $\PP(\N)$. Thus we are in the situation of Proposition~\ref{fb}~$\rm (iii)-\rm (v)$. Since $X$ admits a contraction of an extremal ray to $Y \cong (\PP^1)^2$, we see that $X  \cong (\PP^1)^2 \times \PP^3$ or $(\PP^1)^2 \times Q^3$. 
\end{proof}

\subsection{Case where $\dim Y=3$ }

\begin{pro} If $\dim Y=3$, then $X \cong \PP(T_{\PP^3})$, $\PP(\S_i)$, $\PP^1 \times (\PP^2)^2$, $\PP^2 \times \PP(T_{\PP^2})$ or $(\PP^1)^3 \times \PP^2$.  
\end{pro}

\begin{proof} According to Proposition~\ref{P}, $\rho_X \leq 3$ if $X$ is not isomorphic to $(\PP^1)^5$, $(\PP^1)^3 \times \PP^2$ or $(\PP^1)^2 \times \PP(T_{\PP^2})$. Since $(\PP^1)^5$ and $(\PP^1)^2 \times \PP(T_{\PP^2})$ have no contractions of extremal rays to $3$-dimensional manifolds, we have $X \cong (\PP^1)^3 \times \PP^2$ or $\rho_X \leq 3$. So it is enough to consider the case where $\rho_X \leq 3$. Then it follows from Lemma~\ref{l}~$\rm (i)$ that $\rho_Y=\rho_Z \leq 2$. By our assumption, we see that $5> \dim Z \geq \dim Y =3$.  

If  $\dim Z=3$, then it follows from Lemma~\ref{l}~$\rm (ii)$ and Proposition~\ref{Br} that $X$ admits two different $\PP^2$-bundle structures. By Proposition~\ref{WO}, $X =\PP(T_{\PP^3})$ or $Y \times_CZ$, where $Y$ and $Z$ are $\PP^2$-bundles over a smooth curve $C$. In the latter case, since $Y$ and $Z$ are projective bundles of $\rho = 2$, it follows from Theorem~\ref{4}~$\rm (iii)$ that $Y \cong Z \cong \PP^1 \times \PP^2$ and $C \cong \PP^1$. Therefore, $X \cong (\PP^1 \times \PP^2) \times_{\PP^1}(\PP^1 \times \PP^2) \cong \PP^1 \times (\PP^2)^2$.     

If  $\dim Z=4$, then Lemma~\ref{l}~$\rm (ii)$ and Proposition~\ref{Br} imply that $X$ is a $\PP^1$-bundle over $\PP^4$, $Q^4$, $\PP^1 \times \PP^3$, $\PP^1 \times Q^3$, $(\PP^2)^2$ or $\PP(\N)$. Therefore we are in the situation of Proposition~\ref{fb}~$\rm (i)-(v)$. Since $X$ admits a contraction of an extremal ray to a $3$-dimensional manifold $Y$, $X$ is $\PP(\S_i)$, $\PP^1 \times (\PP^2)^2$ or $\PP^2 \times \PP(T_{\PP^2})$.

\end{proof}

\subsection{Case where $\dim Y=4$ }

\begin{pro} If $\dim Y=4$, then $X$ is isomorphic to one of the following:\\
$(\PP^1)^5$, $(\PP^1)^3 \times \PP^2$, $(\PP^1)^2 \times \PP(T_{\PP^2})$, $(\PP^1)^2 \times \PP^3$, $(\PP^1)^2 \times Q^3$, $\PP^1 \times \PP(\N)=\PP^1 \times \PP(\S)$, $\PP^2 \times \PP(T_{\PP^2})$. 
\end{pro}

\begin{proof}
According to Proposition~\ref{P}, $\rho_X \leq 3$ if $X$ is not isomorphic to $(\PP^1)^5$, $(\PP^1)^3 \times \PP^2$ or $(\PP^1)^2 \times \PP(T_{\PP^2})$. So it is enough to consider the case where $\rho_X \leq 3$. Then it is equivalent to $\rho_Y=\rho_Z \leq 2$.
Lemma~\ref{l}~$\rm (ii)$ and Proposition~\ref{Br} imply that $X$ admits two different $\PP^1$-bundle structures over $4$-folds $Y$ and $Z$ of $\rho \leq 2$. By Lemma~\ref{l}~$\rm (ii)$, $Y$ and $Z$ are $\PP^4$, $Q^4$, $\PP^1 \times \PP^3$, $\PP^1 \times Q^3$, $(\PP^2)^2$ or $\PP(\N)$. Therefore we are in the situation of Proposition~\ref{fb}~$\rm (i)-(v)$. Since $X$ admits two different $\PP^1$-bundle structures over $4$-folds $Y$ and $Z$ of $\rho \leq 2$, $X$ is $(\PP^1)^2 \times \PP^3$, $(\PP^1)^2 \times Q^3$, $\PP^1 \times \PP(\N)=\PP^1 \times \PP(\S)$ or $\PP^2 \times \PP(T_{\PP^2})$.

\end{proof}

\section{Case where $\rho_X=1$}\label{1}

Finally, we deal with Fano manifolds with nef tangent bundles of $\rho_X=1$. All the results in this section are well-known for experts. 

\begin{them}\label{} Let $X$ be a smooth Fano $n$-fold with nef tangent bundle of $\rho_X=1$. Then the pseudoindex $i_X$ satisfies $3 \leq i_X \leq n+1$. Furthermore, the following holds.
\begin{enumerate}
\item If $i_X=n+1$, then $X$ is $\PP^n$.
\item If $i_X=n$, then $X$ is $Q^n$.
\item If $i_X=3$, then $X$ is $\PP^2$, $Q^3$ or $K(G_2)$, where $K(G_2)$ is the $5$-dimensional contact homogeneous manifold of type $G_2$.
\end{enumerate}
\end{them}

\begin{proof} By virtue of Lemma~\ref{pi}, we see that $2 \leq i_X$. Furthermore, it follows from the argument as in \cite[Before Theorem~4.3, P. 623]{Hwang} that $i_X$ is not $2$. On the other hand, if $i_X \geq n+1$, then $X$ is $\PP^n$. This is dealt in \cite{CMSB}. If $i_X=n$, then our assertion follows from \cite{Mi}. The case where $i_X=3$ is treated in \cite[Theorem~4.3]{Hwang}.  

\end{proof}

As a consequence, we have the following:

\begin{cor}\label{} Let $X$ be a smooth Fano $5$-fold with nef tangent bundle of $\rho_X=1$. Then one of the following holds.
\begin{enumerate}
\item $X$ is $\PP^5$, $Q^5$ or $K(G_2)$.
\item $i_X=4$
\end{enumerate}
\end{cor} 

\begin{rem} \rm Let $X$ be a smooth Fano $5$-fold with nef tangent bundle of $\rho_X=1$. For the ample generator $H$ of ${\rm Pic}(X)$, if there exists a rational curve $l$ such that $H.l=1$, then we see that the Fano index coincides with the pseudoindex $i_X=4$. Hence, it turns out that $X$ is a Fano $5$-fold with index $4$. In other words, $X$ is a del Pezzo $5$-fold. 

On the other hand, a rational homogeneous manifold of $\rho=1$ contains a line (see for instance \cite[V.1.15]{Ko}). Furthermore, we see that there is no rational homogeneous $5$-fold of $\rho=1$ with $i_X=4$.
\end{rem} \
{\bf Acknowledgements}
The author would like to thank Dr. Kazunori Yasutake for reading this paper and his comments. He also would like to express his gratitude to referees for their careful reading of the text and useful suggestions and comments. The author is partially supported by the Grant-in-Aid for Research Activity Start-up $\sharp$24840008 from the Japan Society for the Promotion of Science.


\begin{thebibliography}{7}

\bibitem{APW} V. Ancona, T. Peternell, J. A. Wi\'sniewski, {\it  Fano bundles and splitting theorems on projective spaces and quadrics}, Pacific J. Math. 163 (1994), no. 1, 17-42.
\bibitem{Borel} A. Borel, Cohomologie des espaces localement compacts d'aprs J. Leray. Lecture Notes in Mathematics, Vol. 2 Springer-Verlag, Berlin-G${\rm \ddot{o}}$ttingen-Heidelberg, 1964.
\bibitem{BCDD} Bonavero, L., Casagrande, C., Debarre, O., Druel, S., {\it Sur une conjecture de Mukai}, Com- ment. Math. Helv. 78 (2003), 601-626.
\bibitem{Casa} C. Casagrande, {\it Quasi-elementary contractions of Fano manifolds}, Compos. Math. 144 (2008), no. 6, 1429-1460.
\bibitem{CP} F. Campana, T. Peternell, {\it Projective manifolds whose tangent bundles are numerically effective}, Math. Ann. 289 (1991), 169-187.
\bibitem{CP2} F. Campana, T. Peternell, {\it 4-folds with numerically effective tangent bundles and second Betti numbers greater than one}, Manuscripta Math. 79 (1993), no. 3-4, 225-238.
\bibitem{CMSB} K. Cho, Y. Miyaoka and N. I. Shepherd-Barron, {\it Characterizations of projective space and applications to complex symplectic manifolds}, in {\it Higher dimensional birational geometry}, Kyoto, 1997, Advanced Studies in Pure Mathematics, vol. 35, (Mathematical Society of Japan, Tokyo, 2002), 1-88. 
\bibitem{DPS} J. P. Demailly, T. Peternell, M. Schneider, {\it  Compact complex manifolds with numerically effective tangent bundles}, J. Algebraic Geom. 3 (1994), no. 2, 295-345.
\bibitem{Gr} Grothendieck, A.: Le groupe de Brauer I, II, III. Exemples et compl{\'e}ments. In: Dix Expos{\'e}s sur la Coho- mologie des Sch{\'e}mas, North-Holland, Amsterdam, 1968, pp.46-188.
\bibitem{Ha} R. Hartshorne, Algebraic geometry. Graduate Texts in Mathematics, No. 52. Springer-Verlag, New York-Heidelberg, 1977.
\bibitem{Hwang} J.M. Hwang, {\it Rigidity of rational homogeneous spaces}, Proceedings of ICM. 2006 Madrid, volume II, European Mathematical Society, 2006, 613-626.
\bibitem{Ko} J. Koll\'ar, {\it Rational curves on algebraic varieties}, Ergebnisse der Mathematik und ihrer Grenzgebiete (3), vol. 32 (Springer, Berlin, 1996).
\bibitem{KMM} J. Koll\'ar, Y. Miyaoka and S. Mori, {\it Rational connectedness and boundedness of Fano manifolds}, J. Differential Geom. 36 (1992), no. 3, 765-779.
\bibitem{Mi} Y. Miyaoka, {\it Numerical characterisations of hyperquadrics}, Complex analysis in several variables-Memorial Conference of Kiyoshi Oka's Centennial Birthday, Advanced Studies of Pure Mathematics, vol. 42, (Mathematical Society of Japan, Tokyo, 2004), 209-235.
\bibitem{Mok} N. Mok, {\it On Fano manifolds with nef tangent bundles admitting 1-dimensional varieties of minimal rational tangents}, Trans. Amer. Math. Soc. 354 (2002), no. 7, 2639-2658.
\bibitem{Mori} S. Mori, {\it Projective manifolds with ample tangent bundles}, Ann. of Math. (2) 110 (1979), no. 3, 593-606. 
\bibitem{NO} C. Novelli, G. Occhetta, {\it Ruled Fano fivefolds of index two}, Indiana Univ. Math. J. 56 (2007), no. 1, 207-241.
\bibitem{OW} G. Occhetta, J. A. Wi\'sniewski, {\it On Euler-Jaczewski sequence and. Remmert-Van de Ven problem for toric varieties}, Math. Z. 241 (2002), 35-44.
\bibitem{OSS} C. Okonek, M. Schneider and H. Spindler, Vector bundles over complex projective space, Progress in Math., vol. 3, Birkh$\rm \ddot{a}$user, Boston, Basel, Stuttgart, 1980.
\bibitem{Ot} G. Ottaviani, {\it Spinor bundles on quadrics}, Trans. Amer. Math. Soc. 307 (1988), no. 1, 301-316. 
\bibitem{Ota} G. Ottaviani, {\it On Cayley bundles on the five-dimensional quadric}, Boll. Un. Mat. Ital. A (7) 4 (1990), 87-100.
\bibitem{PS} K. H. Paranjape, V. Srinivas, {\it Self-maps of homogeneous spaces}, Invent. Math. 98 (1989), no. 2, 425-444.
\bibitem{SolW}  L. Sol\'a Conde, J. A. Wi\'sniewski, {\it On manifolds whose tangent bundle is big and 1-ample}, Proc. London Math. Soc. (3) 89 (2004), no. 2, 273-290.
\bibitem{SW}  M. Szurek, J. A. Wi\'sniewski, {\it Fano bundles over $\PP^3$ and $Q^3$}, Pacific J. Math. 141 (1990), no. 1, 197-208.
\end{thebibliography}
\end{document}